\documentclass{amsart}

\usepackage{amssymb,enumerate}
\usepackage{eucal,enumerate,amssymb,amsthm,amsmath}

 \newtheorem{thm}{Theorem}[section]
 \newtheorem{cor}[thm]{Corollary}
 
 \newtheorem{prop}[thm]{Proposition}
 \theoremstyle{definition}
 \newtheorem{defn}[thm]{Definition}
 \theoremstyle{remark}

 \numberwithin{equation}{section}

\DeclareMathOperator{\Dom}{{\mathcal{D}}}
\DeclareMathOperator{\Lin}{{\mathcal{L}}}
\DeclareMathOperator{\Real}{Re}
\DeclareMathOperator{\Image}{Im}
\DeclareMathOperator{\R}{{\mathbb{R}}}
\DeclareMathOperator{\const}{const}

\newcommand{\A}{{\mathbf A}}
\newcommand{\vectx}{{\mathbf x}}
\newcommand{\vecty}{{\mathbf y}}
\newcommand{\Hil}{{\mathcal H}}
\newcommand{\semigroup}{{\mathcal T}}

\begin{document}
%-------------------------------------------------------------------------
% editorial commands: to be inserted by the editorial office
%
%\firstpage{1}
%\volume{228}
%\DOI{003-0001}
%\seriesextraline{This is the Concrete Title of this Book\br H.E. R and S.T.C. W, Eds.}
%\firstpage{1}
\title[Exponential decay of semigroups for second-order linear differential equations]
 {Exponential decay of semigroups for second order non-selfadjoint linear differential equations}
%----------Author 1
\author[Nikita Artamonov]{Nikita Artamonov}

\address{%
Dept. for Econometric and Mathematical methods in economics,\\  Moscow State Institute of International Relations (University)\\
119454 av. Vernadskogo 76,  Moscow\\
Russia}

\email{nikita.artamonov@gmail.com}


\thanks{This paper is supported by the Russian Foundation of Basic Research (project No 08-01-00595)}


\date{February 26, 2010}

%\dedicatory{To Anastasia}
%%% -------

\begin{abstract}
The Cauchy problem for second order linear differential equation
 \[
    u''(t)+Du'(t)+Au(t)=0
 \]
in Hilbert space $H$ with a sectorial operator $A$ and an accretive operator $D$ is studied. Sufficient conditions for
exponential decay of the solutions are obtained.
\end{abstract}

%%% ----------------------------------------------------------------------
\maketitle
%%% ----------------------------------------------------------------------
%\tableofcontents

Many linearized equations of mechanics and mathematical physics can
be reduced to a linear differential equation
%In this paper we consider a linear second order differential equation
 \begin{equation}\label{MainEquation}
     u''(t)+Du'(t)+Au(t)=0, %\quad t\geq 0
 \end{equation}
where $u(t)$ is a vector-valued function in an appropriate (finite or infinite dimensional) Hilbert space $H$, $D$ and $A$
are linear (bounded or unbounded) operators on $H$. Properties of the differential equation \eqref{MainEquation} are
closely connected with spectral properties of a quadric pencil
 \[
     L(\lambda)=\lambda^2+\lambda D+A,\quad \lambda\in\mathbb{C}
 \]
which is obtained by substituting exponential functions $u(t)=\exp(\lambda t)x$, $x\in H$ into \eqref{MainEquation}.
In many applications $A$ is a self-adjoint positive definite operator, $D$ is a self-adjoint positive definite or an
accretive operator  (see definition in section 1). In this case the differential equation \eqref{MainEquation} and spectral properties of
the related quadric pencil $L(\lambda)$ are well-studied, see \cite{ChenTriggiani, HrynivShkalikov1, HrynivShkalikov2,Huang, JacobTrunk1, JacobTrunk2, JacobMorrisTrunk, JacobTrunkWinklmeier, Markus}
and references therein. It was obtained a localization of the pencil's spectrum, sufficient
conditions of the completeness of eigen- and adjoint vectors of the pencil $L(\lambda)$ and it was proved, that all solutions of \eqref{MainEquation} exponentially decay. The exponential decay means, that the total energy  exponentially
decreases and corresponding mechanical system is stable. In paper
\cite{Shkalikov} was studied spectral properties of the pencil $L(\lambda)$ for a self-adjoint non-positive definite
operator $A$ and an accretive operator $D$.

But some models of continuous mechanics are reduced to differential equation \eqref{MainEquation} with sectorial
operator $A$, see \cite{AglazinKiiko1, IlyushinKiiko1, Artamonov1} and references therein. In this cases methods,
developed for self-adjoint operator $A$, cannot be applied.

The aim of this paper is the study of a Cauchy problem for second-order linear differential equation \eqref{MainEquation}
in a Hilbert space $H$ with initial conditions
 \begin{equation}\label{CauchyProblem}
  u(0)=u_0\quad u'(0)=u_1.
 \end{equation}
The shiffness operator $A$ is assumed to be a sectorial operator,  the damping operator $D$ is assumed to
be  an accretive operator.

By $\Lin(H',H'')$ denote a space of bounded operators acting from a Hilbert space $H'$ to a Hilbert
space $H''$. $\Lin(H)=\Lin(H,H)$ is an algebra of bounded operators acting on  Hilbert space $H$.

\section{Preliminary results}

First let us recall some definitions \cite{Haase, Kato}.
\begin{defn}
Linear operator $B$ with dense domain $\Dom(B)$ is called \textit{accretive} if
$\Real(Bx,x)\geq0$ for all $x\in\Dom(B)$ and \textit{m-accretive}, if the range of operator $B+\omega I$
is dense in $H$ for some $\omega>0$.
\end{defn}
An accretive operator $B$ is m-accretive iff $B$ has not accretive extensions \cite{Kato}. For m-accretive operator
 \[
     \rho(B)\supset\{\lambda\in\mathbb{C}\;:\; \Real\lambda< 0\}.
 \]

\begin{defn}
An accretive operator $B$ is called \textit{sectorial} or \textit{$\omega$-accretive} if for some $\omega\in[0,\pi/2)$
 \[
     \bigl|\Image(Bx,x)\bigr|\leq \tan(\omega) \Real(Bx,x)\quad x\in\Dom(B).
 \]
If a sectorial operator has not sectorial extensions, then it's called \textit{m-sectorial} or
\textit{m-$\omega$-accretive}.
\end{defn}
The sectorial property means that the numerical range of the operator $B$  belongs to a sector
 \[
    \{z\in\mathbb{C}\;|\; |\Image z|\leq \tan(\omega)\Real z\}.
 \]
For a sectorial operator $B$ there exist \cite{Kato} a self-adjoint non-negative operator $T_B$ and a self-adjoint
operator $S_B\in\Lin(H)$, $\|S_B\|\leq\tan(\omega)$ such that
 \[
    \Real(B x,x)=(T_B^{1/2}x,T_B^{1/2}x), \quad
     B\subset T_B^{1/2}(I+iS_B)T_B^{1/2}
 \]
and $B=T_B^{1/2}(I+iS_B)T_B^{1/2}$ iff $B$ is m-sectorial.

Throughout this paper we will assume, that
\begin{itemize}%[a)]
  \item[(A)] Operator $A:\Dom(A)\subset H\to H$ is m-sectorial and for some positive $a_0$
  \[
     \Real(Ax,x)\geq a_0(x,x)\quad x\in\Dom(A).
  \]
\end{itemize}
Since $A$ is m-sectorial there exist a self-adjoint positive definite operator $T$ and a
self-adjoint $S\in\Lin(H)$, such that
 \begin{gather*}
     \Real(Ax,x)=(T^{1/2}x,T^{1/2}x)\geq a_0(x,x),\quad x\in\Dom(A) \\
     A=T^{1/2}(I+iS)T^{1/2}.
  \end{gather*}
The operator  $A$ is invertible and
 \[
    A^{-1}=T^{-1/2}(I+iS)^{-1}T^{-1/2}.
 \]
By $H_s$ $(s\in\R)$ denote a collection of Hilbert spaces generated by a self-adjoint operator $T^{1/2}$:
\begin{itemize}
\item for $s\geq 0$ $H_s=\Dom(T^{s/2})$  endowed with a norm $\|x\|_s=\|T^{s/2}x\|$;
\item for $s<0$ $H_s$ is a closure of $H$ with respect to the norm $\|\cdot\|_s$.
\end{itemize}
Obviously $H_0=H$. The operator $T^{1/2}$ can be considered now as an unitary operator mapping $H_{s}$ on $H_{s-1}$.
$A$ is a bounded operator  $A\in\Lin(H_2,H_0)$ and it can be extended to a bounded
operator $\tilde{A}\in\Lin(H_1,H_{-1})$. The inverse operator $A^{-1}$ can be extended to a bounded
operator $\tilde{A}^{-1}\in\Lin(H_{-1},H_{1})$.

 By $(\cdot,\cdot)_{-1,1}$ denote a duality pairing  on $H_{-1}\times H_{1}$. Note, that for all $x\in H_{-1}$ and
$y\in H_1$ we have
  \[
     \Bigl|(x,y)_{-1,1}\Bigr|\leq \|x\|_{-1}\cdot\|y\|_1
  \]
and $(x,y)_{-1,1}=(x,y)$ if $x\in H$. Further,
 \[
     \Real (\tilde{A}x,x)_{-1,1}=(Tx,x)_{-1,1}=(T^{1/2}x,T^{1/2}x)=\|x\|^2_1,\quad x\in H_{1}=\Dom(T^{1/2}).
 \]
Denote $\tilde{S}=T^{1/2}ST^{1/2}\in\Lin(H_{1},H_{-1})$. Then, for the operator $\tilde{A}$ we have
a representation $\tilde{A}=T+i\tilde{S}$ and
 \[
   \Image(\tilde{A}x,x)_{-1,1}=(\tilde{S}x,x)_{-1,1}\quad x\in H_1.
 \]
Also $(\tilde{S}x,y)_{-1,1}=\overline{(\tilde{S}y,x)}_{-1,1}$ for all $x,y\in H_1$.

Following paper \cite{JacobTrunk2} we will assume
\begin{itemize}
 \item[(B)] $D$ is a bounded operator $D\in\Lin(H_{1},H_{-1})$, and
  \begin{equation}\label{ConditionBeta}
       \beta=\inf_{x\in H_{1},x\ne 0}\frac{\Real(Dx,x)_{-1,1}}{\|x\|^2}>0.
  \end{equation}
\end{itemize}
Operator $T^{-1/2}$ is an unitary operator mapping $H_s$ on
$H_{s+1}$, therefore an operator $D'=T^{-1/2}DT^{-1/2}$, acting on
$H$, is bounded. Let
 \[
    D_1=\frac12T^{1/2}\Bigl(D'+(D')^*\Bigr)T^{1/2}\quad
    D_2=\frac{1}{2i}T^{1/2}\Bigl(D'-(D')^*\Bigr)T^{1/2},
 \]
Obviously $D_1,D_2\in\Lin(H_{1},H_{-1})$,  $D=D_1+iD_2$ and for all $x\in H_1$
 \[
     \Real(Dx,x)_{-1,1}=(D_1x,x)_{-1,1}\geq\beta\|x\|^2,
     \quad \Image(Dx,x)_{-1,1}=(D_2x,x)_{-1,1}.
 \]
Also $(D_j x,y)_{-1,1}=\overline{(D_j y,x)}_{-1,1}$ for all $x,y\in H_1$ ($j=1,2$).

\section{Main result}

\begin{defn}
A vector-valued function $u(t)\in H_1$ is called a solution of the differential equation
\eqref{MainEquation} if $u'(t)\in H_1$, $u''(t)\in H$, $Du'(t)+\tilde{A}u(t)\in H$ and
 \begin{equation}\label{MainEquationGen}
     u''(t)+Du'(t)+\tilde{A}u(t)=0
 \end{equation}
\end{defn}
If $u(t)$ is a solution of \eqref{MainEquationGen}, then a vector-function
 \[
    \vectx(t)=\begin{pmatrix} u'(t) \\ u(t) \end{pmatrix}
 \]
(formally) satisfies a first-order differential equation
 \begin{equation}\label{FirstOrderEqn}
    \vectx'(t)= \A \vectx(t)
 \end{equation}
with a block operator matrix
 \[
     \A=\begin{pmatrix}
       -D &  -\tilde{A} \\
        I & 0
    \end{pmatrix}.
 \]
From mechanical viewpoint it is most natural to consider the equation \eqref{FirstOrderEqn} in an
"energy" space $\Hil=H\times H_{1}$ with a dense domain of the operator $\A$
\cite{HrynivShkalikov1, HrynivShkalikov2, JacobTrunk2, Shkalikov}
 \[
    \Dom(\A)=\left\{\left.\begin{pmatrix} x_1 \\ x_2 \end{pmatrix} \right|
     x_1,x_2\in H_1, \; Dx_1+\tilde{A}x_2\in H
     \right\}.
 \]
An inverse of $\A$ is formally defined by a block operator matrix
 \[
     \A^{-1}=\begin{pmatrix}
       0 & I \\
       -\tilde{A}^{-1}  & -\tilde{A}^{-1}D
      \end{pmatrix}.
 \]
Let $\vecty=(y_1,y_2)^\top\in\Hil=H\times H_{1}$, then
 \[
     \A^{-1}\vecty=\begin{pmatrix}
       y_2 \\ -\tilde{A}^{-1}y_1-\tilde{A}^{-1}Dy_2
     \end{pmatrix}=\begin{pmatrix}
       x_1 \\ x_2
     \end{pmatrix}.
 \]
Since  $\tilde{A}^{-1}\in\Lin(H_{-1},H_1)$ and $D\in\Lin(H_{1},H_{-1})$, then
$\tilde{A}^{-1}D\in\Lin(H_1,H_1)$. Therefore
$-\tilde{A}^{-1}y_1-\tilde{A}^{-1}Dy_2\in H_1$ and $\A^{-1}\vecty\in H_{1}\times H_1$.
Moreover,
 \[
    D x_1+\tilde{A}x_2=D y_2+\tilde{A}\left(-\tilde{A}^{-1}y_1-\tilde{A}^{-1}Dy_2\right)
    =-y_1\in H.
 \]
Thus $\A^{-1}\vecty\in\Dom(\A)$. Since $I\in\Lin(H_1,H)$ the operator $\A^{-1}$ is bounded and therefore
the operator $\A$ is closed and $0\in\rho(\A)$.

Let $(\vectx,\vecty)_{\Hil}$ be a natural scalar product on $\Hil=H\times H_1$ and
$\|\vectx\|_{\Hil}^2=(\vectx,\vecty)_{\Hil}$.

If operator $A$ is self-adjoint, the spectral properties of operator $\A$ are well-studied: $-\A$ is
an m-accretive operator in the Hilbert space $\Hil=H\times H_1$ (see
\cite{ChenTriggiani,  HrynivShkalikov1, HrynivShkalikov2,  Huang, JacobTrunk1, JacobTrunk2} and references therein) and,
consequently, $\A$ is a generator of a $C_0$-semigroup. Thus,  differential equation \eqref{FirstOrderEqn} (and equation \eqref{MainEquationGen}) is correctly solvable in
the space $\Hil$ for all $\vectx(0)=(u_1,u_0)^\top\in\Dom(\A)$. Moreover, in this case operator $\A$ is a generator of a
contraction semigroup \cite{HrynivShkalikov2}. It implies, that all solutions of \eqref{FirstOrderEqn} (and
\eqref{MainEquationGen}) exponentially decay, i.e. for some $C,\omega>0$
 \[
     \|\vectx(t)\|_{\Hil}\leq C\exp(-\omega t)\|\vectx(0)\|_{\Hil}\quad t\geq 0.
 \]
For non-selfadjoint $A$ operator ($-\A$) is not longer accretive in
the space $\Hil$ with respect to the standard scalar product. But,
under some assumptions, one can define a new scalar product on
$\Hil$, which is topologically equivalent to the given one, such
that an operator ($-\A-qI$) (for some $q\geq0$) is m-accretive and
therefore the operator $\A$ is a generator of a $C_0$-semigroup on
$\Hil$. If $q>0$, then $\A$ is a generator of a contraction
semigroup and all solutions of \eqref{FirstOrderEqn} exponentially
decay.

Let $k\in(0,\beta)$ ($\beta$ is defined by \eqref{ConditionBeta}). Consider on the space $\Hil$ a sesquilinear form
 \begin{multline*}
    [\vectx,\vecty]_{\Hil}=\\ (T^{1/2}x_2,T^{1/2}y_2)+k(D_1x_2,y_2)_{-1,1}-k^2(x_2,y_2)+
    (x_1+kx_2,y_1+ky_2),\\
    %(Tx_2,y_2)_{-1,1}+k(D_1x_2,y_2)_{-1,1}-k^2(x_2,y_2)+
   %(x_1+kx_2,y_1+ky_2),\\
  \vectx=(x_1,x_2)^\top,\; \vecty=(y_1,y_2)^\top\in\Hil.
 \end{multline*}
Obviously, $[\vectx,\vecty]=\overline{[\vecty,\vectx]}$ and
 \[
   [\vectx,\vectx]_{\Hil}=\|x_2\|_1^2+k(D_1x_2,x_2)_{-1,1}+\|x_1\|^2+2k\Real(x_1,x_2).
 \]
Since $(D_1x,x)_{-1,1}=\Real(Dx,x)_{-1,1}\geq\beta\|x\|^2$ and
 \[
  2|\Real(x_1,x_2)|\leq2|(x_1,x_2)|\leq 2\|x_1\|\cdot\|x_2\|\leq
  \frac{\|x_1\|^2}{\beta}+\beta\|x_2\|^2,
 \]
then
 \begin{multline*}
   [\vectx,\vectx]_{\Hil}\geq \|x_2\|^2_1+k\Bigl((D_1x,x)_{-1,1}-\beta\|x_2\|^2\Bigr)+\left(1-\frac{k}{\beta}\right)\|x_1\|^2\geq\\
  \|x_2\|^2_{1}+\left(1-\frac{k}{\beta}\right)\|x_1\|^2.
 \end{multline*}
Inequatlities\footnote{$\|D_1\|$ is a norm of operator $D_1\in\Lin(H_{1},H_{-1})$, i.e.
$\|D_1\|=\sup_{x\in H_1,x\ne0}\|D_1x\|_{-1}/\|x\|_1$}
$|(D_1x,x)_{-1,1}|\leq \|D_1x\|_{-1}\cdot\|x\|_1\leq \|D_1\|\cdot\|x\|^2_1 $  and $\|x\|^2_1\geq a_0\|x\|^2$ imply
 \begin{multline*}
    [\vectx,\vectx]_{\Hil}\leq \Bigl(1+k\|D_1\|\Bigr)\|x_2\|^2_1+k\beta \|x_2\|^2+\left(1+\frac{k}{\beta}\right)\|x_1\|^2 \\
   \leq\left(1+k\|D_1\|+\frac{k\beta}{a_0}\right)\|x_2\|^2_1+\left(1+\frac{k}{\beta}\right)\|x_1\|^2.
 \end{multline*}
Thus,
 \[
    \left(1-\frac{k}{\beta}\right)\|\vectx\|^2_{\Hil}\leq [\vectx,\vectx]_{\Hil}
    \leq \const \|\vectx\|^2_{\Hil}
 \]
and $[\cdot,\cdot]_{\Hil}$ is a scalar product on $\Hil$, which is topologically equivalent to the given one.
Denote $|\vectx|^2_{\Hil}=[\vectx,\vectx]_{\Hil}$.

%Denote
% \begin{align*}
%   Q_1(\vectx,\vecty) &=(T^{1/2}x_2,T^{1/2}y_2)+k(D_1x_2,y_2)_{-1,1}-k^2(x_2,y_2) \\
%   Q_2(\vectx,\vecty) &=(x_1+kx_2,y_1+ky_2) \\
%   [\vectx,\vecty]_{\Hil} & = Q_1(\vectx,\vecty)+Q_2(\vectx,\vecty)
% \end{align*}
%Since $(D_1x,x)=\Real(Dx,x)_{-1,1}\geq\beta(x,x)$ and $(D_1x,x)_{-1,1}\leq c\|x\|_1^2$ $(c>0)$ for all $x\in H_1$ we have
% \begin{equation}\label{Q1}
%     \|x_2\|^2_1+k(\beta-k)\|x_2\|^2\leq Q_1(\vectx,\vectx)\leq c\|x_2\|^2_1
% \end{equation}
%Obviously, $Q_2(\vectx,\vectx)\geq0$ for all $\vectx\in\Hil$. Therefor
%By definition
% \[
%     Q_2(\vectx,\vectx)=(x_1,x_1)+k^2(x_2,x_2)+2k\Real(x_1,x_2)
% \]
%Using inequality $2|(x_1,x_2)|\leq (\|x_1\|^2/\beta+\beta\|x_2\|^2)$ one obtain the estimate
% \begin{equation}\label{Q2}
%   \left(1-\frac{k}{\beta}\right)\|x_1\|^2+k(k-\beta)\|x_2\|^2\leq Q_2(\vectx,\vectx)\leq
%   \left(1+\frac{k}{\beta}\right)\|x_1\|^2+k(k+\beta)\|x_2\|^2
% \end{equation}
%Summing inequalities \eqref{Q1} and \eqref{Q2} we obtain
% \[
%    \left(1-\frac{k}{\beta}\right)\|x_1\|^2+\|x_2\|_1^2\leq[\vectx,\vecty]_{\Hil}\leq
%    \left(1+\frac{k}{\beta}\right)\|x_1\|^2+k(k+\beta)\|x_2\|^2+c\|x_2\|^2_1
% \]
%Since $\|x_2\|\leq\const \|x_2\|_1$ we finally obtain the inequality
% \[
%    \left(1-\frac{k}{\beta}\right)
% \]

\begin{thm}\label{MainTheorem1}
Let the assumptions (A) and (B) hold and for some $k\in(0,\beta)$ and $m\in(0,1]$
 \begin{equation}\label{Omega1}
    \omega_1=\inf_{x\in H_1, x\ne0}\frac{\frac{1}{k}(D_1x,x)_{-1,1}-\|x\|^2-
    \frac{1}{4m}\bigl\|(\frac{1}{k}\tilde{S}-D_2)x\bigr\|_{-1}}
    {\|x\|^2}\geq0.
 \end{equation}
Then the operator $\A$ is a generator of a $C_0$-semigroup $\semigroup(t)=\exp\{t\A\}$ $(t\geq0)$ and
 \[
     \bigl\|\semigroup(t)\bigr\|_{\Hil}\leq \const\cdot \exp(-t k\theta)
 \]
where
  \[
      \theta=\min\left\{\frac{\omega_1}{2}, \frac{1-m}{\omega_2}\right\}\geq 0
  \]
and\footnote{Obviously, $\omega_2\leq 1+k\|D_1\|+k^2/a_0$}
 \begin{equation}\label{Omega2}
     \omega_2=\sup_{x\in H_1,x\ne0} \frac{\|x\|^2_1+k(D_1x,x)_{-1,1}+k^2\|x\|^2}{\|x\|^2_1}
  \end{equation}
\end{thm}
\begin{proof}
For $\vectx=(x_1,x_2)^\top\in\Dom(\A)$ let us consider a quadric form
 \begin{multline*}
    [\A\vectx,\vectx]_{\Hil}=(T^{1/2}x_1,T^{1/2}x_2)+k(D_1x_1,x_2)_{-1,1}-k^2(x_1,x_2)+\\
    (-Dx_1-\tilde{A}x_2+kx_1,x_1+kx_2)=\\
   (Tx_1,x_2)_{-1,1}+k(D_1x_1,x_2)_{-1,1}-(Dx_1,x_1)_{-1,1}\\-(\tilde{A}x_2,x_1)_{-1,1}
   +k(x_1,x_1)-k(Dx_1,x_2)_{-1,1}-k(\tilde{A}x_2,x_2)_{-1,1}=\\
   -(Dx_1,x_1)_{-1,1}+k(x_1,x_1)-k(\tilde{A}x_2,x_2)_{-1,1}-ik(D_2x_1,x_2)_{-1,1}+\\
   (T x_1,x_2)_{-1,1}-(Tx_2,x_1)_{-1,1}-i(\tilde{S}x_2,x_1)_{-1,1}
 \end{multline*}
We used decompositions $\tilde{A}=T+i\tilde{S}$ and $D=D_1+iD_2$. Consequently,
\begin{multline*}
 \Real[\A\vectx,\vectx]_{\Hil}=-(D_1x_1,x_1)_{-1,1}+k(x_1,x_1)-k(Tx_2,x_2)_{-1,1}-\\
 \Real\Bigl(ik(D_2x_1,x_2)_{-1,1}+i(\tilde{S}x_2,x_1)_{-1,1}\Bigr)=\\
 -(D_1x_1,x_1)_{-1,1}+k\|x_1\|^2-k\|x_2\|^2_1-\\
\Image\left((\tilde{S}x_1,x_2)_{-1,1}-k(D_2x_1,x_2)_{-1,1}\right)
 \end{multline*}
and
 \[
    -\frac{1}{k}\Real[\A\vectx,\vectx]_{\Hil}=\frac{1}{k}(D_1x_1,x_1)_{-1,1}-\|x_1\|^2+
   \|x_2\|^2_1+\Image\left(\left(\frac{1}{k}\tilde{S}-D_2\right)x_1,x_2\right)_{-1,1}.
 \]
Since
 \begin{multline*}  %\left|\Image\left((\tilde{S}-kD_2)x,x\right)_{-1,1}\right|\leq
     \left|\left(\left(\frac{1}{k}\tilde{S}-D_2\right)x_1,x_2\right)_{-1,1}\right|\leq
      \left\|\left(\frac{1}{k}\tilde{S}-D_2\right)x_1\right\|_{-1}\cdot\|x_2\|_{1}\leq\\
    \frac{1}{4m}\left\|\left(\frac{1}{k}\tilde{S}-D_2\right)x_1\right\|^2_{-1}+m\|x_2\|^2_1,
 \end{multline*}
then
 \begin{multline*}
  -\frac{1}{k}\Real[\A\vectx,\vectx]_{\Hil}\geq \frac{1}{k}(D_1 x_1,x_1)_{-1,1}-\|x_1\|^2-
   \frac{1}{4m}\left\|\left(\frac{1}{k}\tilde{S}-D_2\right)x_1\right\|^2_{-1}+\\
  (1-m)\|x_2\|^2_1\geq \omega_1\|x_1\|^2+(1-m)\|x_2\|^2_1.
 \end{multline*}
Further, an inequality
 \[
    2k|\Real(x_1,x_2)|\leq 2|(x_1,kx_2)|\leq 2\|x_1\|\cdot\|kx_2\|\leq \|x_1\|^2+k^2\|x_2\|^2
 \]
implies
 \begin{equation}\label{NormEstimate}
    [\vectx,\vectx]_{\Hil}\leq 2\|x_1\|^2+\|x_2\|_1^2+k(D_1x_2,x_2)_{-1,1}+k^2\|x_2\|^2\leq
    2\|x_1\|^2+\omega_2\|x_2\|^2_1.
 \end{equation}
Thus
 \[
     -\frac{1}{k}\Real[\A\vectx,\vectx]_{\Hil}\geq \omega_1\|x_1\|^2+(1-m)\|x_2\|^2_1\geq
     \theta(2\|x_1\|^2+\omega_2\|x_2\|^2_1)\geq \theta[\vectx,\vectx]_{\Hil}
 \]
and an operator ($-\A-k\theta I$) is accretive. Moreover, the operator ($-\A-k\theta I$) is
 m-accretive (since $0\in\rho(\A)$) and\footnote{Obviously, the operator ($-\A$) is m-accretive as well.}
 \[
   \rho\left(-\A-k\theta I\right)\subset\{\lambda\in\mathbb{C},\;\Real\lambda <0\}\Rightarrow
   \rho(-\A)\supset\{\lambda\in\mathbb{C},\;\Real\lambda <k\theta\}.
 \]
Therefore, the operator $\A$ is a generator of a $C_0$-semigroup \cite{Haase, HillePhillips} $\semigroup(t)=\exp\{t\A\}$, $t\geq 0$ and
 \[
     \bigl|\semigroup(t)\bigr|_{\Hil}\leq \exp(-k\theta t),\quad t\geq 0.
 \]
On the space $\Hil$ norms $|\vectx|_{\Hil}$ and
$\|\vectx\|_{\Hil}$ are equivalent and the inequality
 \[
     \bigl\|\semigroup(t)\bigr\|_{\Hil}\leq \const\cdot \exp(-k\theta t),\quad t\geq 0
 \]
holds for some positive constant.
\end{proof}
\begin{cor}
Under the conditions of the theorem \ref{MainTheorem1} for all $\vectx_0=(u_1,u_0)^\top\in\Dom(\A)$
vector-function
 \[
     \vectx(t)=\begin{pmatrix} w(t) \\ u(t) \end{pmatrix}=\semigroup(t)\vectx_0\in\Dom(\A)
 \]
satisfies the first order differential equation \eqref{FirstOrderEqn}. $u(t)$
satisfies the second-order differential equation \eqref{MainEquationGen} with the initial conditions
\eqref{CauchyProblem} and an inequality
 \[
    \|u(t)\|_1^2+\|u'(t)\|^2\leq \const\cdot \exp\{-2k\theta t\}\Bigl(\|u_0\|^2_1+\|u_1\|^2\Bigr)
 \]
holds for all $t\geq 0$.
\end{cor}

Consider now a more strong assumption on the operator $D$:
\begin{itemize}
\item[(C)] $D\in\Lin(H_{1},H_{-1})$ and
 \[
    \delta=\inf_{x\in H_{1},x\ne 0}\frac{\Real(Dx,x)_{-1,1}}{\|x\|^2_1}>0.
 \]
\end{itemize}
It is easy to show that the assumption (C) implies (B) and $\beta>a_0\delta$.

By $\|\tilde{S}\|$  and  $\|D_2\|$ denote norms of the bounded operators $\tilde{S}\in\Lin(H_1,H_{-1})$
and $D_2\in \Lin(H_1,H_{-1})$. Then for all $x\in H_1$
   \[
      \|\tilde{S}x\|_{-1}\leq \|\tilde{S}\|\cdot\|x\|_1,\quad\|D_2x\|_{-1}\leq \|D_2\|\cdot\|x\|_1
   \]

\begin{thm}\label{MainTheorem2}
Let the assumptions (A) and (C) are fulfilled and for some $k\in(0,\beta)$ and some $p,q>0$ with $p+q\leq1$
 \[
     \omega_1'=a_0\left(\frac{\delta}{k}-\frac1{4pk^2}\|\tilde{S}\|^2
     -\frac{1}{4q}\|D_2\|^2\right)\geq1
 \]
Then the operator $\A$ is a generator of a $C_0$-semigroup $\semigroup(t)=\exp\{t\A\}$ $(t\geq0)$ and
 \[
     \bigl\|\semigroup(t)\bigr\|_{\Hil}\leq \const\cdot \exp(-t k\theta')
 \]
where
  \[
      \theta'=\min\left\{\frac{\omega_1'-1}{2}, \frac{1-p-q}{\omega_2}\right\}\geq 0
  \]
and $\omega_2$ is defined by \eqref{Omega2}.
%\footnote{Obviously, $\omega_2\leq 1+k\|D_1\|_{-1,1}+k^2/a_0$}
% \[
%     \omega_2=\sup_{x\in H_1,x\ne0} \frac{\|x\|^2_1+k(D_1x,x)_{-1,1}+k^2\|x\|^2}{\|x\|^2_1}
%  \]
\end{thm}
\begin{proof}
Consider on Hilbert space $\Hil=H\times H_{1}$ the scalar product $[\vectx,\vecty]_{\Hil}$. Then
 \begin{multline*}
    -\frac{1}{k}\Real[\A\vectx,\vectx]_{\Hil}=\frac{1}{k}(D_1x_1,x_1)_{-1,1}-\|x_1\|^2+
   \|x_2\|^2_1+\\ \frac{1}{k}\Image(\tilde{S}x_1,x_2)_{-1,1}-\Image(D_2x_1,x_2)_{-1,1}
 \end{multline*}
(see the proof of the theorem \ref{MainTheorem1}). Since
 \begin{multline*}
  |\Image(D_2x_1,x_2)_{-1,1}|\leq|(D_2x_1,x_2)_{-1,1}| \leq \|D_2 x_1\|_{-1}\cdot\|x_2\|_1\leq \\
  \frac{1}{4q}\|D_2x_1\|^2_{-1}+q\|x_2\|^2_1\leq \frac{1}{4q}\|D_2\|^2\cdot\|x_1\|^2_{1}+q\|x_2\|^2_1
 \end{multline*}
 \begin{multline*}
  \frac1k|\Image(\tilde{S}x_1,x_2)_{-1,1}|\leq|(\frac1k\tilde{S}x_1,x_2)_{-1,1}| \leq
  \left\|\frac1k\tilde{S} x_1\right\|_{-1}\cdot\|x_2\|_1\leq \\
  \frac{1}{4p}\left\|\frac1k\tilde{S}x_1\right\|^2_{-1}+p\|x_2\|^2_1\leq
  \frac{1}{4pk^2}\|\tilde{S}\|^2\cdot\|x_1\|^2_{1}+p\|x_2\|^2_1
 \end{multline*}
and taking into account $(D_1x,x)_{-1,1}\geq\delta \|x\|^2_1$ and $\|x\|^2_1\geq a_0\|x\|^2$
we obtain
 \begin{multline*}
    -\frac{1}{k}\Real[\A\vectx,\vectx]_{\Hil}\geq \frac{1}{k}(D_1x_1,x_1)_{-1,1}-\|x_1\|^2-\\
    \frac{\|\tilde{S}\|^2}{4pk^2}\cdot\|x_1\|^2_{1}-
    \frac{\|D_2\|^2}{4q}\cdot\|x_1\|^2_{1}+(1-p-q)\|x_2\|^2_1\geq\\
    \left(\frac{\delta}{k}-\frac{\|\tilde{S}\|^2}{4pk^2}-\frac{\|D_2\|^2}{4q}\right)\|x_1\|^2_1-
    \|x_1\|^2+(1-p-q)\|x_2\|^2_1\geq \\
    (\omega_1'-1)\|x_1\|^2+(1-p-q)\|x_2\|^2_1.
 \end{multline*}
Using \eqref{NormEstimate} we finally have
 \[
    -\frac{1}{k}\Real[\A\vectx,\vectx]_{\Hil}\geq \theta' [\vectx,\vectx]_{\Hil}.
 \]
Thus an operator ($-\A-k\theta' I$) in m-accretive (since $0\in\rho(\A)$) and
 \[
   \rho(-\A)\supset\{\lambda\in\mathbb{C},\;\Real\lambda <k\theta'\}.
 \]
Therefore, the operator $\A$ is a generator  of a $C_0$-semigroup \cite{Haase, HillePhillips}
$\semigroup(t)=\exp\{t\A\}$ ($t\geq 0$) and
 \[
     \bigl|\semigroup(t)\bigr|_{\Hil}\leq \exp(-k\theta' t),\quad t\geq 0.
 \]
Since the norms $|\vectx|_{\Hil}$ and
$\|\vectx\|_{\Hil}$ are equivalent then we have an inequality
 \[
     \bigl\|\semigroup(t)\bigr\|_{\Hil}\leq \const\cdot \exp(-k\theta' t),\quad t\geq 0
 \]
for some positive constant.
\end{proof}

\begin{cor}
Under the conditions of the theorem \ref{MainTheorem2} for all $\vectx_0=(u_1,u_0)^\top\in\Dom(\A)$ a
vector-valued function
 \[
     \vectx(t)=\begin{pmatrix} w(t) \\ u(t) \end{pmatrix}=\semigroup(t)\vectx_0\in\Dom(\A)
 \]
satisfies the first order differential equation
\eqref{FirstOrderEqn}. $u(t)$ satisfies the second-order
differential equation \eqref{MainEquationGen} with an initial
conditions \eqref{CauchyProblem} and the inequality
 \[
    \|u(t)\|_1^2+\|u'(t)\|^2\leq \const\cdot \exp\{-2k\theta' t\}\Bigl(\|u_0\|^2_1+\|u_1\|^2\Bigr)
 \]
holds for all $t\geq 0$.
\end{cor}

\section{Related spectral problem}

Let us consider a quadric pencil associated with the differential equation \eqref{MainEquation}
 \[
     L(\lambda)=\lambda^2I+\lambda D+A\quad \lambda\in\mathbb{C}.
 \]
Since $D:H_1\to H_{-1}$ it is more naturally to consider an extension of pencil
 \[
    \tilde{L}(\lambda)=\lambda^2I+\lambda D+\tilde{A}
 \]
mapping $H_1$ to $H_{-1}$. Moreover, $\tilde{L}(\lambda)\in\Lin(H_1,H_{-1})$ for all
$ \lambda\in\mathbb{C}$.
\begin{defn}
The resolvent set of the pencil $\tilde{L}(\lambda)$ is defined as
 \[
    \rho(\tilde{L})=\{\lambda\in\mathbb{C}\;:\; \exists \tilde{L}^{-1}(\lambda)\in\Lin(H_{-1},H_1)\}
 \]
The spectrum of the pencil is  $\sigma(\tilde{L})=\mathbb{C}\backslash\rho(\tilde{L})$.
\end{defn}
In \cite{HrynivShkalikov2,Shkalikov} it was proved that $\sigma(\tilde{L})=\sigma(\A)$ and for $\lambda\ne0$
 \[
     (\A-\lambda I)^{-1}=\begin{pmatrix}
      \lambda^{-1}\left(\tilde{L}^{-1}(\lambda)\tilde{A}-I\right) & -\tilde{L}^{-1}(\lambda) \\
      \tilde{L}^{-1}(\lambda)\tilde{A} & -\lambda\tilde{L}^{-1}(\lambda)
     \end{pmatrix}
 \]
This result allows to obtain a localization of the pencil's spectrum in a half-plane.
\begin{prop}
\textbf{1.} Under the conditions of the theorem \ref{MainTheorem1} the spectrum of the pencil $\tilde{L}(\lambda)$ belongs to
a half-plane
 \[
    \sigma(\tilde{L})\subseteq\{\Real\leq -k\theta\}.
 \]
\textbf{2.} Under the conditions of the theorem \ref{MainTheorem2} the spectrum of the pencil $\tilde{L}(\lambda)$ belongs to
a half-plane
 \[
    \sigma(\tilde{L})\subseteq\{\Real\leq -k\theta'\}.
 \]
\end{prop}

Acknowledgement: the author thanks Prof. Carsten Trunk for fruitful discussions during IWOTA 2009 and
Prof. A.A. Shkalikov for discussions.

%\begin{thebibliography}{1}
%\bibitem{test} A. B. C. Test, \textit{On a Test.} J. of Testing
%\textbf{88} (2000), 100--120.
%\bibitem{latex} G. Gr\"atzer, \textit{Math into \LaTeX.} 3rd Edition,
%Birkh\"auser, 2000.
%\end{thebibliography}

\end{document}